\documentclass{article}
\usepackage{amsfonts}
\usepackage{amsmath}
\usepackage{sectsty}
\usepackage{latexsym}
\usepackage{amsbsy}
\usepackage{amssymb}
\usepackage{verbatim}
\usepackage{amsthm}
\usepackage{amsfonts}
\usepackage{amsmath}
\usepackage{makeidx}
\usepackage{graphicx}
\usepackage{layout}
\usepackage{pb-diagram}
\usepackage{amscd}
\usepackage{anysize}
\usepackage{tikz}
\usepackage{graphicx}
\usepackage{upgreek}
\usepackage[all,cmtip]{xy}
\usepackage{youngtab}
\usepackage{hyperref}
\usepackage{color}
\usepackage{mathrsfs}

\newcommand{\beq}{\begin{equation*}}
\newcommand{\eeq}{\end{equation*}}
\newcommand{\bthm}{\begin{theorem}}
\newcommand{\ethm}{\end{theorem}}
\newcommand{\blm}{\begin{lemma}}
\newcommand{\elm}{\end{lemma}}
\newcommand{\Z}{\mathbb{Z}}
\newcommand{\C}{\mathbb{C}}

\newtheorem{theorem}{Theorem}[section]

\newtheorem{corollary}[theorem]{Corollary}

\newtheorem{lemma}[theorem]{Lemma}

\newtheorem{problem}[theorem]{Problem}
\newtheorem{proposition}[theorem]{Proposition}
\newtheorem{remark}[theorem]{Remark}

\newcommand{\den}{\mathrm{den}}

\title{\Large\bf Free circle actions on $(n-1)$-connected $(2n+1)$-manifolds}
\author{Yi Jiang \and Yang Su }
\newcommand{\Addresses}{{
  \bigskip
  \footnotesize

  \textsc{Yi Jiang}, \textsc{Academy for Multidisciplinary Studies, Beijing National Center for Applied Mathematics, Capital Normal
  University, Beijing, 100048, China}\par\nopagebreak
 \texttt{jiangyi@cnu.edu.cn}
  \medskip

  \textsc{Yang Su}, \textsc{School of Mathematics and Systems Science, Chinese Academy of Sciences, Beijing, 100190, China, AND 
  School of Mathematical Sciences, University of Chinese Academy of Sciences, Beijing, 100049, China}\par\nopagebreak
  \texttt{suyang@math.ac.cn}
  \medskip
}}
\date{}

\begin{document}
\maketitle

\begin{abstract}
In this paper, we determine those $(n-1)$-connected $(2n+1)$-manifolds with torsion free homology that admit free circle actions up to almost diffeomorphism, provided that $n\equiv5,7 \mod 8$.
\end{abstract}

\section{Introduction}\label{se:Introduction}
In this paper, all manifolds under consideration are smooth, closed and oriented. We are interested in the problem \emph{when does a manifold $M$ admit a free smooth action by the unit circle $S^1$} (free circle action for short in the sequel). 
There are studies of this problem for certain families of manifolds by various authors. For instance, when $M$ is a homotopy sphere, this problem was studied in  \cite{H66,L68,MY68,S71}; 
when $M$ is an $(n-1)$-connected $(2n+1)$-manifold, this problem has been solved for $n=2$ (\cite{DL05}) and $n=3$ (\cite{J14}). For further examples see \cite{GL71,L72,D22,GR23}.

In this note we consider this problem for $(n-1)$-connected $(2n+1)$-manifolds with torsion free homology when $n >3$. The classification of $(n-1)$-connected $(2n+1)$-manifolds with torsion free homology up to almost diffeomorphism was obtained by the work of several authors. Recall that two $n$-manifolds $M_1$ and $M_2$  are \emph{almost diffeomorphic} if there is a homotopy $n$-sphere $\Sigma$ such that the connected sum $M_1\#\Sigma$ is diffeomorphic to $M_2$ (cf.\cite[p.223]{CR17}).

\begin{theorem}(\cite[Theorem 7]{W67}, \cite[Theorem 3]{Wi72}, \cite[Theorem 1.10]{SZ22}, \cite[Theorem B]{C02})\label{classification}
Let $M$ be an $(n-1)$-connected $(2n+1)$-manifold with torsion free homology. When $n \equiv 5 \pmod 8$,  $M$ is almost diffeomorphic to the connected sum $\#_r(S^n\times S^{n+1})$ of $r$ copies of $S^n\times S^{n+1}$ with $r > 0$. When $n \equiv 7 \pmod 8$, $M$ is almost diffeomorphic to $\#_{r}(S^n\times S^{n+1}) \# X_{l(M)}$ with $r\geq 0$.
\end{theorem}

Here we explain the notations in the above theorem. For $n\equiv 7 \mod 8$, let $l(M)$ be the divisibility of the $\frac{n+1}{4}$-th Pontrjagin class $p_{\frac{n+1}{4}}(M)$, i.e., $p_{\frac{n+1}{4}}(M)$ is $l(M)$ multiple of a primitive element with $l(M) \ge 0$. It is shown by Kervaire \cite[Lemma 1.1]{K59} that $p_{\frac{n+1}{4}}(M)=(\frac{n-1}{2})!\beta(M)$, where $\beta(M)\in H^{n+1}(M)$ is the only obstruction to trivializing the stable tangent bundle of $M$. Therefore 
$l(M)$ is always divisible by $(\frac{n-1}{2})!$. Moreover, for $n>7$ and for any natural number $l$ divisible by $(\frac{n-1}{2})!$, there exists an $(n-1)$-connected $(2n+1)$-manifold $X_l$, unique up to almost diffeomorphism, such that  $H^{n+1}(X_l) \cong \Z$ and $l(X_l)=l$. The manifold $X_l$ can be taken as the linear $S^n$-bundle over $S^{n+1}$ with Euler class $0$ and the $\frac{n+1}{4}$-th Pontryagin class $l$ multiple of a generator of $H^{n+1}(S^{n+1})$.  For example, $X_0 = S^n \times S^{n+1}$.The case $n=7$ is similar except  that the obstruction $\beta(M)$ is an even multiple of a primitive element, due to the existence of Hopf invariant $1$ elements in this dimension,  and the manifold $X_l$ exists for each natural number $l$ divisible by $2 \cdot (\frac{7-1}{2})!=12$. See the Appendix for more details.

As an application of this classification, we determine those $(n-1)$-connected $(2n+1)$-manifolds with torsion free homology that admit free circle actions up to almost diffeomorphism, provided that $n\equiv5,7 \mod 8$. 

\begin{theorem}\label{mainthn=5mod8}
Let $M$ be an $(n-1)$-connected $(2n+1)$-manifold where $H_n(M)$ is free and $n\equiv 5 \mod 8$. Then there exists a homotopy sphere $\Sigma$ such that $M\#\Sigma$ admits a free circle action.
\end{theorem}

Recall that the Bernoulli numbers $B_1, B_2, \cdots$ are  the coefficients in the expansion
$$\frac{z}{e^z-1}=1-\frac{z}{2}+\frac{B_1}{2!}z^2-\frac{B_2}{4!}z^4+\frac{B_3}{6!}z^6-\cdots,$$
(cf.~\cite[Appendix B]{MS74}). For any rational number $r$ let $\den(r)$ denote the denominator of $r$ expressed as a fraction in lowest terms.

\begin{theorem}\label{mainthn=7mod8}
Let $M$ be an $(n-1)$-connected $(2n+1)$-manifold where $H_n(M)$ is free and $n\equiv 7 \mod 8$. Then there exists a homotopy sphere $\Sigma$ such that $M\#\Sigma$ admits a free circle action if and only if one of the following conditions holds:
\begin{enumerate}
  \item[(1)] $M$ is almost diffeomorphic to $\#_{2r}(S^n\times S^{n+1})$ for some nonnegative integer $r$;
  \item[(2)] $M$ is almost diffeomorphic to $\#_{2r}(S^n\times S^{n+1})\# X_l$ for some nonnegative integer $r$ and some nonnegative integer $l$ divisible by $(\frac{n-1}{2})!\den(\frac{B_{(n+1)/4}}{n+1})$.
\end{enumerate}
\end{theorem}

\begin{remark}
In the above theorem, condition (1) is equivalent to that the $n$-th Betti number $b_n(M)$ of $M$ is even and $l(M)=0$; condition (2) is equivalent to that $b_n(M)$ is odd and $l(M)$ is divisible by $(\frac{n-1}{2})!\den(\frac{B_{(n+1)/4}}{n+1})$.
\end{remark}

In Section \ref{se:case n=5mod8} we prove Theorem \ref{mainthn=5mod8} by a direct geometric construction.  To prove Theorem \ref{mainthn=7mod8}, one observes  that a manifold $M$ admits a free circle action if and only if it is the total space of an $S^1$-bundle over a manifold $N$, and $N$ is the orbit space of the action. The key ingredient in the proof is the analysis of the topological invariants of the orbit space. In Section \ref{se:coh of orbit} and Section \ref{se:P-class of orbit}
 we study the cohomology ring and the Pontrjagin class of the orbit space, respectively. Finally we prove Theorem \ref{mainthn=7mod8}  in Section \ref{se:case n=7mod8}. 
 
\subsection*{Acknowledgements}
The authors would like to thank Haibao Duan for bringing the topic to their attention, and thank Matthias Kreck for helpful conversations. Y. Jiang is partially supported by NSFC 11801298 and NSFC 12371070. Y. Su is partially supported by NSFC 12071462.

\section{A construction of free circle actions}\label{se:case n=5mod8}
In this section we recall a construction of free circle actions and prove Theorem \ref{mainthn=5mod8}. 

Given two $m$-manifolds $N_1$ and $N_2$, the connected sum $N_1\# N_2$ is constructed as follows (c.f.~\cite[Section 2]{MK63}). Taking embeddings
$$i_k:D^m\to N_k, k=1,2$$
such that $i_1$ preserves the orientation and $i_2$ reserves the orientation. We obtained $N_1\# N_2$ from the disjoint sum $N_1\backslash i_1(0)+N_2\backslash i_2(0)$ by identifying $i_1(tu)$ with $i_2((1-t)u)$ for each unit vector $u\in S^{m-1}$ and each $0<t<1$. Now given oriented $S^1$-bundles $E_k\to N_k,k=1,2.$ A new $S^1$-bundle over $N_1\# N_2$ is constructed as follows. Taking embeddings $j_k:S^1\times D^m\to E_k, k=1,2$ such that $j_k$ covers $i_k$. The \emph{$S^1$-connected sum} $E_1\#_{S^1}E_2$ is an $(m+1)$-manifold obtained from the disjoint sum $E_1\backslash j_1(S^1\times 0)+E_2\backslash j_2(S^1\times0)$ by identifying $j_1(x,tu)$ with $j_2(x,(1-t)u)$ for each pair $(x,u)\in S^1\times S^{m-1}$ and each $0<t<1$. The bundle projections $E_k\to N_k, k=1,2$ induce an $S^1$-bundle projection $E_1\#_{S^1}E_2\to N_1\# N_2$ (c.f.~\cite[Section 3]{HS13}).

Let $m\geq 3$ and let $\alpha:S^1\to SO(m)$ represent the generator of $\pi_1SO(m)=\Z_2$. Let $\tau:S^1\times D^m\to S^1\times D^m$ be the homeomorphism given by $\tau(t,x)=(t,\alpha(t)x).$ Take an embedding $f:D^m\to N$ for an $m$-manifold $N$. Let $\Sigma_0N$ and $\Sigma_1N$ be the result manifolds obtained by surgery along the embeddings
$$id_{S^1}\times f: S^1\times D^m\to S^1\times N,$$
$$(id_{S^1}\times f)\circ\tau: S^1\times D^m\to S^1\times N,$$
respectively. 

\begin{lemma}(\cite[Theorem B and Proposition 3.2]{D22}) \label{lm:DuanThB}
  Let $E\to B$ be an $S^1$-bundle over an $m-$manifold with $E$ simply-connected and $m\geq 4$. Then for any simply-connected $m$-manifold $N$, the $S^1$-connected sum $E\#_{S^1}(S^1\times N)$ is diffeomorphic to $E\#\Sigma_0N$ if $B$ is nonspin and to $E\#\Sigma_1N$ if $B$ is spin. 

There are diffeomorphisms 
$$\Sigma_0(S^p\times S^q) \cong \Sigma_1(S^p\times S^q) \cong (S^p\times S^{q+1})\# (S^{p+1}\times S^q)$$
for any $p\leq q$ with $q\geq 3$
\end{lemma}

\begin{proposition}\label{Lm:connected sum admits circle action}
Let $E$ be a simply-connected $(2n+1)$-manifold and $n\geq 3$. If $E$ admits a free circle action, then $\#_2 (S^n\times S^{n+1})\# E$ admits a free circle action.
\end{proposition}
\begin{proof}Let $B$ be the orbit space of a free circle action on $E$. Then $E$ is the total space of a circle bundle over $B$, and $E\#_{S^1}(S^1\times S^n\times S^n)$ is the total space of a circle bundle over $B\# (S^n\times S^n)$. By Lemma \ref{lm:DuanThB}, $E\#_{S^1}(S^1\times S^n\times S^n)$ is diffeomorphic to either $E\#\Sigma_0(S^n\times S^n)$ or $E\#\Sigma_1(S^n\times S^n)$, and hence is diffeomorphic to $\#_2 (S^n\times S^{n+1}) \# E$ by the second paragraph of Lemma \ref{lm:DuanThB}. This proves the Proposition.
\end{proof}

\begin{corollary}\label{Co:circle action on exotic sum}
  Let $r,n$ be positive integers and $n\geq 3$. If $\Sigma$ is a homotopy $(2n+1)$-sphere admitting a free circle action, then $\#_{2r}(S^n\times S^{n+1})\#\Sigma$ admits a free circle action.
\end{corollary}

Note that there are many exotic $(2n+1)$-spheres admitting free circle actions. In fact, the set of homotopy $(2n+1)$-spheres admitting free circle actions has been determined for $n=3,4,5,6$ by \cite{MY68} and \cite[p.402, Theorem I.10(i)]{B71}.

\begin{proof}[Proof of Theorem \ref{mainthn=5mod8}]By Theorem \ref{classification} there is a homotopy sphere $\Sigma$ such that $M\#\Sigma$ is diffeomorphic to $\#_r(S^n\times S^{n+1})$ where $r$ is the rank of $H_n(M)$. Hence it suffices to show that $\#_r(S^n\times S^{n+1})$ admits a free circle action. This is a direct consequence of Corollary \ref{Co:circle action on exotic sum}
 and the fact that $S^{2n+1}$ and $S^n\times S^{n+1}$ admit free circle actions with orbit spaces $\C P^n$ and $\C P^{\frac{n-1}{2}}\times S^{n+1}$ respectively.
\end{proof}

\section{Cohomology of the orbit space}\label{se:coh of orbit}
If a $(2n+1)$-manifold $M$ admits a free circle action, then it is the total space of the corresponding circle bundle over the orbit space $N$ which is a $2n$-manifold. Conversely, the total space $M$ of a circle bundle over a $2n$-manifold $N$ is a $(2n+1)$-manifold which admits a free circle action. In this section we analyze the cohomology ring of $N$ when $M$ is $(n-1)$-connected with torsion free homology. In the next section we study the Pontrjagin classes of $N$.

\blm\label{Lm:cohomology of orbit}
Let $n$ be an odd integer greater than $1$. Let $S^1\times M\to M$ be a free circle action on an $(n-1)$-connected $(2n+1)$-manifold $M$ with torsion free homology. Then the orbit space $N$ is a simply-connected $2n$-manifold whose cohomology ring $H^*(N)$ is isomorphic to 
  $H^*(\#_r(S^n\times S^n)\# \C P^n)$ or $H^*(\#_r(S^n\times S^n)\# (\C P^{\frac{n-1}{2}}\times S^{n+1}))$.
  
Conversely, every simply-connected $2n$-manifold $N$ with
  $$H^*(N)\cong H^*(\#_r(S^n\times S^n)\# \C P^n)\mbox{ or } H^*(\#_r(S^n\times S^n)\# (\C P^{\frac{n-1}{2}}\times S^{n+1}))$$
  can be realized as the orbit space of some free circle action on an $(n-1)$-connected $(2n+1)$-manifold with torsion free homology.
\elm

\begin{proof}
To prove the lemma, it suffices to verify that for a circle bundle $S^1 \to M \to N$, the manifolds $M$ and $N$ satisfy the conditions stated in the lemma. The fundamental groups of $M$ and $N$ are related in the homotopy exact sequence of the circle bundle
\begin{equation}\label{fundamental}
\pi_1(S^1) \to \pi_1(M) \to \pi_1(N) \to 0.
\end{equation}
The cohomology groups of $M$ and $N$ are related in the Gysin sequence 
\begin{equation}\label{longGysin}
  \xymatrix{H^{j-2}(N)\ar[r]^-{- \cup \, t}&H^j(N)\ar[r]&H^j(M)\ar[r]&H^{j-1}(N)\ar[r]^-{- \cup \, t}&H^{j+1}(N)},
\end{equation}
 where $t \in H^2(N)$ is the Euler class of the circle bundle.
 
 Now assume that $M$ is $(n-1)$-connected $(2n+1)$-manifold with torsion free homology. The exact sequence (\ref{fundamental}) implies that $N$ is simply connected. The cohomology ring $H^*(N)$ is computed from the Gysin sequence 
\begin{equation}\label{Gysin}
  \xymatrix{H^{i-1}(M)\ar[r]&H^{i-2}(N)\ar[r]^-{- \cup \,t}&H^i(N)\ar[r]&H^i(M)}.
\end{equation}
Since $M$ is $(n-1)$-connected, the exact sequence (\ref{Gysin}) implies that for $0\leq i\leq n-1$,
$$H^i(N)=\left\{
           \begin{array}{ll}
             \Z t^{\frac{i}{2}}, & \hbox{if $i$ is even;} \\
             0, & \hbox{otherwise,}
           \end{array}
         \right.
$$
and for $n+3\leq i\leq 2n$, the homomorphism $- \cup t \colon H^{i-2}(N) \to H^i(N)$ is an isomorphism.
Moreover, the Gysin sequence implies that there is a monomorphism $H^n(N)\to H^n(M)$. Since $M$ is $(n-1)$-connected, this implies $H^n(M)$ is free and hence $H^n(N)$ free. Note that the rank of $H^n(N)$ must be even, because the intersection form of $N$ is nondegenerate and skew-symmetric since $n$ is odd. Now since $H^{n-1}(N)=\Z t^{\frac{n-1}{2}}$ and $H^{n+1}(N)\cong H_{n-1}(N)\cong\Z$, it remains to show the homomorphism $- \cup t \colon H^{n-1}(N)\to H^{n+1}(N)$ must be either an isomorphism or a trivial map. This is easily obtained by the Gysin sequence 
$$\xymatrix{H^{n-1}(N)\ar[r]^-{- \cup \,t}&H^{n+1}(N)\ar[r]&H^{n+1}(M)}$$
and the fact that $H^{n+1}(M)\cong H_n(M)$ is free.
 
Conversely, assume $N$ is simply-connected,  the cohomology ring $H^*(N)$ is isomorphic to either $H^*(\#_r(S^n\times S^n)\#\C P^n)$ or $H^*(\#_r(S^n\times S^n)\#(\C P^{\frac{n-1}{2}}\times S^{n+1}))$ for some $r\geq 0$, and the Euler class $t$ is a generator  of $H^2(N)\cong \Z$. The exact sequence (\ref{fundamental}) implies that $\pi_1M$ is abelian and hence  $\pi_1M\cong H_1(M)$. To prove that $M$ is $(n-1)$-connected with torsion free homology, it suffices to show that $H_i(M)$ is trivial for $1\leq i\leq n-1$ and $H_n(M)$ is torsion free. This follows from the exact sequence (\ref{longGysin}).
When $j=1$, the sequence (\ref{longGysin}) implies that $H^1(M)=0$ since $H^1(N)=0$ and $- \cup t\colon H^0(N) \to H^2(N)$ is an isomorphism. When $2\leq j\leq n-1$, the sequence (\ref{longGysin}) implies that $H^j(M)=0$ because $H^{n-2}(N)=0$ and $- \cup t \colon H^{i-2}(N) \to H^i(N)$ is an isomorphism for $2\leq i\leq n-1$. When $j=n,n+1$, the sequence (\ref{longGysin}) implies that $H^j(M)$ is free because the cohomology groups of $N$ are free, the homomorphism $- \cup t \colon H^i(N)\to H^{i+2}(N)$ is trivial for $i=n-2,n$ and $- \cup t \colon H^{n-1}(N)\to H^{n+1}(N)$ is either an isomorphism or a trivial map. This shows that $H_i(M)$ is trivial for $1\leq i\leq n-1$ and $H_n(M)\cong H^{n+1}(M)$ is torsion free.
\end{proof}

\section{Pontrjagin classes of the orbit spaces}\label{se:P-class of orbit}

We have seen in the last section that the cohomology ring of the orbit space $N$ of a free circle action on $M$ is isomorphic to $H^*(\#_r(S^n \times S^n)\# \C P^n)$ or $H^*(\#_r(S^n \times S^n) \#  (\C P^{\frac{n-1}{2}}\times S^{n+1}))$. It will be shown in the proof of Theorem \ref{mainthn=7mod8} that in the first case, the Pontrjagin class $p_{\frac{n+1}{4}}(M)$ vanishes. In order to determine the Pontrjagin class $p_{\frac{n+1}{4}}(M)$ in the second case,  we analyze the Pontrjagin class of $N$ in this section. Since we are in the case $n \equiv 7 \pmod 8$, we assume $n=4k-1$. Though under our assumption $k$ is even, Lemma 4.1 and Lemma 4.2 hold also for $k$ odd.

\blm\label{Lm: Restriction of Pontrjagin class}
Let $N$ be a simply connected $(8k-2)$-manifold whose cohomology ring is  isomorphic to $H^*(\#_r(S^{4k-1}\times S^{4k-1})\# (\C P^{2k-1}\times S^{4k}))$ with $r\geq0$ and $k\geq 1$. If the Pontrjagin class $p_{k}(N)$ is $d(N)$ multiple of a primitive element, then $d(N)$ is divisible by $(2k-1)! \cdot \den(\frac{B_k}{4k})$. 
\elm
\begin{proof}
Let $\{\hat{A}_k(p_1,\cdots,p_k)\}$ be the multiplicative sequence of polynomials with
$$\frac{\sqrt{t}/2}{\sinh(\sqrt{t}/2)}=1+\sum_{n=1}^{\infty}(-1)^n\frac{(2^{2n}-2)B_n}{2^{2n}((2n)!)}t^{n}$$
as characteristic power series (cf.\cite[\S1]{H95}). Let $\alpha_k$ be the coefficient of $p_k$ in $\hat{A}_k(p_1,\cdots,p_k)$.  We first show the following two facts:
\begin{enumerate}
  \item [(1)] $\alpha_k \cdot d(N)$ is an integer.
  \item [(2)] The integer $d(N)$ is divisible by $a_k \cdot (2k-1)!$, where $a_k=1$ if $k$ is even,  $a_k=2$ if $k$ is odd. 
\end{enumerate}

The proof of (1) is based on the integrality of twisted $\hat A$-genera. Let $ch(\eta) \in H^*(N;\mathbb Q)$ be the Chern character of a virtual complex vector bundle $\eta$ over $N$, $d\in H^2(N)$ be a cohomology class whose mod $2$ reduction is the second Stiefel-Whitney class of $N$, $\hat{A}(N)= \sum_{i=1}^{\infty}\hat{A}_i(p_1(N),\cdots,p_i(N))$ be the $\hat A$-class of $N$, and $[N] \in H_{2n}(N)$ be the fundamental class of $N$, then  by  \cite[Theorem 26.1.1]{H95}, $\langle ch(\eta) \cdot e^{d/2}\cdot\hat{A}(N),[N]\rangle$ is an integer. Now let $L$ be the complex line bundle over $N$ with first Chern class $c_1(L)=t$, a generator of $H^2(N)$. Take $\eta = (L-1)^{2k-1}$. By the assumption of $H^*(N)$, 
\begin{eqnarray*}
\langle ch(\eta) \cdot e^{d/2}\cdot\hat{A}(N),[N]\rangle & = & \langle(e^t-1)^{2k-1}e^{d/2}\hat{A}(N),[N]\rangle \\
  &=& \langle t^{2k-1}e^{d/2}\hat{A}(N),[N]\rangle \\
  &=& \langle t^{2k-1}\hat{A}(N),[N]\rangle \\
  & = & \langle t^{2k-1} (\alpha_k p_k(N) +\cdots), [N]\rangle \\
  & = & \pm \alpha_k d(N).
\end{eqnarray*}
This proves (1).

To prove (2), let $\pi:N_t\to N$ be the $S^1$-bundle over $N$ with Euler class $t$, a generator of $H^2(N)$. In the segment of the Gysin sequence
$$\xymatrix{H^{4k-2}(N)\ar[r]^-{- \cup \, t}&H^{4k}(N)\ar[r]^-{\pi^*}&H^{4k}(N_t)\ar[r]&H^{4k-1}(N)\ar[r]^-{- \cup \, t}&H^{4k+1}(N)},$$
the two homomorphisms $- \cup \, t$ are trivial, 
therefore we have a split short exact sequence  
\begin{equation}\label{shortGysin}
  0 \to \xymatrix{H^{4k}(N)\ar[r]^-{\pi^*}&H^{4k}(N_t)\ar[r]&H^{4k-1}(N)} \to 0.
\end{equation}
Furthermore,  the tangent bundle $TN_t$ of $N_t$ is isomorphic to the Whitney sum $\pi^*TN \oplus V$, where $V = \ker d\pi$ is the vertical bundle of the circle bundle. The orientable line bundle $V$ is trivial, therefore $p_{k}(N_t)=\pi^*p_{k}(N)$ and hence $l(N_t)=d(N)$ by the sequence (\ref{shortGysin}). By Lemma \ref{Lm:cohomology of orbit} $N_t$ is a $(4k-1)$-connected $(8k-1)$-manifold. Since $\pi_{4k-2}SO=0$, it is shown in \cite[Lemma 1.1]{K59} that $l(N_t)$ is divisible by $a_k(2k-1)!$. This proves (2).

Now by (2), there is an integer $\beta$ such that
$d(N)=a_k(2k-1)!\beta$. Since $\alpha_k=-\frac{B_k}{2 \cdot (2k)!}$(cf.\cite[3.4]{BH60}), we have
$$\alpha_k d(N)=-\frac{B_k}{2 \cdot (2k)! }a_k(2k-1)!\beta=-\frac{B_ka_k\beta}{4k}$$
is an integer. It follows that $\beta$ must be divisible by $\den(\frac{B_ka_k}{4k})$ and hence $d(N)=a_k(2k-1)!\beta$ must be a multiple of $a_k(2k-1)!\den(\frac{B_ka_k}{4k})$.
Now it suffices to show that $\den(\frac{a_kB_k}{4k})=\den(\frac{B_k}{4k})/a_k$. If $k$ is even,  then $a_k=1$ and this is clearly true. When $k$ is odd, $a_k=2$. It is known that $\den(B_k)$ is even (cf.\cite[Appendix B, Theorem B.3]{MS74}), hence the numerator of $B_k$ is odd. From this it is easy to see $\den(\frac{2B_k}{4k})= \den(\frac{B_k}{4k})/2$.
\end{proof}

\blm\label{Lm:realization}Let $a_k=1$ when $k$ is even, and $a_k=2$ when $k$ is odd.
For any integer $d$ divisible by $a_k \cdot (2k-1)! \cdot \den(\frac{B_k}{4k})$, there exists a $(8k-2)$-manifold $N$ homotopy equivalent to $\#_r(S^{4k-1}\times S^{4k-1})\# (\C P^{2k-1}\times S^{4k})$, such that the Pontrjagin class $p_k(N)$ is $d$ multiple of a primitive element.
\elm
The proof of this lemma uses surgery theory. We recall some elementary notions here (cf.~\cite[p.109]{W99}, \cite[pp.45-46]{B72}). Let $BO_k$ be the classifying space of the orthogonal group $O_k$. Let $G_k$ be the topological monoid of self-homotopy equivalences of $S^{k-1}$ and let $BG_k$ be its classifying space. The natural maps $G_k\to G_{k+1}$ and $O_k\to O_{k+1}$ induce maps $BG_k\to BG_{k+1}$ and $BO_k\to BG_{k+1}$, respectively. Then $BO = \lim_{k \to \infty} BO_k$ is the classifying space of stable vector bundles and $BG= \lim_{k \to \infty}BG_k$ is the classifying space of stable spherical fibrations. By taking the sphere bundle associated to a real vector bundle one has a forgetful map $BO\to BG$. The induced homomorphism $J \colon \pi_i(BO) \to \pi_i(BG)$ is the $J$-homomorphism. 
\begin{proof}
Let $\xi$ be a stable vector bundle over $S^{4k}$, whose sphere bundle is trivial as a stable spherical fibration. This means that $\xi$ is in the kernel of the $J$-homomorphism $J \colon \pi_{4k}(BO) \to \pi_{4k}(BG)$. The group $\pi_{4k}(BO)$ is isomorphic to $\mathbb Z$ and $\ker J$ is a subgroup of index $\den(\frac{B_k}{4k})$ (first proved by Adams \cite[Theorem 3.7]{Ada65} assuming the Adams conjecture \cite[Conjecture 1.2]{Ada63}, which was later proved by Quillen \cite{Q71}). It is shown by Kervaire \cite[Lemma 1.1]{K59} that the image of the homomorphism $p_k \colon \pi_{4k}(BO) \to \mathbb Z$, $\xi \mapsto \langle p_k(\xi), [S^{4k}] \rangle $ is a subgroup of index $a_k \cdot (2k-1)!$.

For an integer $d$  divisible by $a_k \cdot (2k-1)! \cdot \den(\frac{B_k}{4k})$, we may choose a vector bundle $\xi$ such that $\langle p_k(\xi), [S^{4k}] \rangle =d$. Since $\xi$ is a vector bundle reduction of the Spivak normal fibration of $S^{4k}$, we may consider the surgery problem 
$$
\xymatrix{
 \nu X \ar[r]^{\bar f} \ar[d] &  \xi \ar[d] \\
 X \ar[r]^f & S^{4k}}
 $$
where $X$ is a closed $4k$-manifold,  $f$ is map of degree $1$, covered by a bundle map $\bar f$ from the stable normal bundle $\nu X$ to $\xi$.  The surgery obstruction, i.~e., the obstruction to doing surgery on $(X,f)$ to get a homotopy equivalence $f' \colon X' \to S^{4k}$, is $\theta(\bar f,f) \in L_{4k}(\mathbb Z) \cong \mathbb Z$, where $L_{4k}(\mathbb Z)$ is the Wall surgery obstruction group. 

Taking the product with $\C P^{2k-1}$ we have a surgery problem 
$$
\xymatrix{
 \nu \C P^{2k-1} \times \nu X \ar[r]^{\mathrm{id} \times \bar f} \ar[d] &  \nu \C P^{2k-1} \times \xi \ar[d] \\
\C P^{2k-1} \times X \ar[r]^{\mathrm{id} \times f} & \C P^{2k-1} \times S^{4k}}
 $$
By the product formula of surgery obstructions (cf.~\cite[p.33]{B72}, \cite[Lemma 13B.4]{W99}), the surgery obstruction of this surgery problem is 
$$\theta(\mathrm{id} \times \bar f, \mathrm{id} \times f) = \chi(\C P^{2k-1}) \cdot \theta(\bar f,f) \in L_{8k-2}(\mathbb Z),$$
where $\chi(\C P^{2k-1})$ is the Euler characteristic. Since $\chi(\C P^{2k-1})=2k$ and $L_{8k-2}(\mathbb Z) \cong \mathbb Z/2$, the surgery obstruction $\theta(\mathrm{id} \times \bar f, \mathrm{id} \times f) $ vanishes. Therefore by surgery we get a homotopy equivalence $g \colon Y \to \C P^{2k-1} \times S^{4k}$, which is covered by a bundle map $\bar g \colon \nu Y \to \nu \C P^{2k-1} \times \xi$. The Pontrjagin class 
$$p_k(Y) = p_k(TY) = g^*p_k(T\C P^{2k-1} \times \xi^{-1}),$$ 
where $T\C P^{2k-1}$ is the stable tangent bundle of $\C P^{2k-1}$ , and $\xi^{-1}$ is the stable inverse of $\xi$. Since $p_k(T\C P^{2k-1} \times \xi^{-1}) = - \pi_2^* (p_k(\xi))$, where $\pi_2 \colon \C P^{2k-1} \times S^{4k} \to S^{4k}$ is the projection, the divisibility of $p_k(Y)$ equals the divisibility of $p_k(\xi)$. Taking connected sum with $r$ copies of $S^{4k-1} \times S^{4k-1}$ we get $N$.
\end{proof}

\section{Proof of Theorem \ref{mainthn=7mod8}}\label{se:case n=7mod8}
\begin{proof}[Proof of Theorem \ref{mainthn=7mod8}]
Let $N$ be the orbit space of a free circle action on $M$. Then $M$ is the total space of a circle bundle $\eta$ over $N$, $\pi \colon M \to N$ with Euler class $t$,  a generator of  $H^2(N)$.  The tangent bundle $TM$  is isomorphic to the Whitney sum $\pi^*TN \oplus V$, where $V = \ker d\pi$ is the vertical bundle of  $\eta$ and is trivial as an orientable line bundle. Therefore
$$p_{\frac{n+1}{4}}(M)=\pi^*p_{\frac{n+1}{4}}(N).$$
By Lemma \ref{Lm:cohomology of orbit},  the cohomology ring $H^*(N)$ is isomorphic to $H^*(\#_r(S^n\times S^n)\#\C P^n)$ or $H^*(\#_r(S^n\times S^n)\#(\C P^{\frac{n-1}{2}}\times S^{n+1}))$.
If $H^*(N)$ is isomorphic to $H^*(\#_r(S^n\times S^n)\#\C P^n)$, then in the Gysin sequence
 $$\xymatrix{H^{n-1}(N)\ar[r]^-{-\cup \,t}&H^{n+1}(N)\ar[r]^-{\pi^*}&H^{n+1}(M)\ar[r]&H^n(N)\ar[r]&0}$$
the homomorphism $-\cup \,t$  is an isomorphism, therefore the homomorphism $\pi^*$ is trivial and hence $$H_n(M)\cong H^{n+1}(M)\cong H^n(N)\cong \Z^{2r}, \ \ p_{\frac{n+1}{4}}(M)=\pi^*p_{\frac{n+1}{4}}(N)=0.$$ 
      If  $H^*(N)$ is isomorphic to $H^*(\#_r(S^n\times S^n)\#(\C P^{\frac{n-1}{2}}\times S^{n+1}))$, then in the Gysin sequence $$\xymatrix{H^{n-1}(N)\ar[r]^-{- \cup \,t}&H^{n+1}(N)\ar[r]^-{\pi^*}&H^{n+1}(M)\ar[r]&H^n(N)\ar[r]&0}$$
the homomorphism $-\cup \,t$ is trivial. Therefore we have a split short exact sequence  
$$0 \to \xymatrix{H^{n+1}(N)\ar[r]^-{\pi^*}&H^{n+1}(M)\ar[r]&H^n(N)} \to 0.$$  
This implies that  $H^{n+1}(M)\cong \Z^{2r+1}$ and the divisibility $l(M)$ of $p_{\frac{n+1}{4}}(M)$ equals to the divisibility $d(N)$ of $p_{\frac{n+1}{4}}(N)$. By Lemma \ref{Lm: Restriction of Pontrjagin class}, the integer $l(M)$ must satisfy the condition in Theorem \ref{mainthn=7mod8}. 

Conversely, the total space of the circle bundle over $\#_r(S^n\times S^n)\#\C P^n$ with Euler class a generator of $H^2(\#_r(S^n\times S^n)\#\C P^n)$ is diffeomorphic to $\#_{2r}(S^n \times S^{n+1})$, as we have seen in the proof of Proposition \ref{Lm:connected sum admits circle action}. On the other hand, for an integer $l$ satisfying the condition in Theorem \ref{mainthn=7mod8}, let $N$ be a manifold given in Lemma 
\ref{Lm:realization} with $d(N)=l$, and take the total space $M$ of the circle bundle  over $N$ with Euler class a generator of $H^2(N)$. Then by the above discussion $H_n(M)$ is isomorphic to $\mathbb Z^{2r+1}$ and $l(M)$ equals $d(N)$. By Theorem \ref{classification} $M$ is diffeomorphic to $\#_{2r}(S^n \times S^{n+1}) \# X_l$.
\end{proof}

\section{Appendix}
In this section we describe the structure of $(n-1)$-connected $(2n+1)$-manifolds with torsion free homology and $n \equiv 3 \pmod 4$ as a twisted double, and give a proof of the divisibility of $p_{\frac{n+1}{4}}(M)$. 

Let $M$ be an $(n-1)$-connected $(2n+1)$-manifold with torsion free homology, let $\{x_1, \cdots, x_r\}$ be a basis of $H_n(M)$, which are represented by embedded spheres $S^n_1, \cdots , S^n_r \subset M$. Assume that $n \equiv 3 \pmod 4$, then since in this dimension $\pi_{n-1}SO(n+1)=0$, the normal bundles of these embedded spheres are trivial. Choose a framing of the normal bundle, we have submanifolds $S^n_i \times D^{n+1} \subset M$ for $i=1, \cdots r$. Connect these submanifolds by tubes we have a submanifold $V_1 = \natural_r (S^n \times D^{n+1}) \subset M$, where $\natural$ stands for the boundary connected sum operation. The homomorphism $H_n(M - \mathring V_1) \to H_n(M)$ induced by the inclusion map $M - \mathring V_1 \to M$ is an isomorphism, therefore the preimages $ \bar x_1, \cdots , \bar x_r \in H_n(M - \mathring V_1)$ of $x_1, \cdots , x_r$ form a basis of $H_n(M - \mathring V_1)$. There are embedded spheres $S^n_1, \cdots, S^n_r \subset M$ representing this basis, and we obtain a submanifold $V_2 = \natural_r(S^n \times D^{n+1}) \subset M-\mathring V_1$. From the construction of $V_1$ and $V_2$, it is easy to see that the homomorphisms $H_*(\partial V_i) \to H_*(M-(\mathring V_1 \cup \mathring V_2))$ induced by the inclusion maps are isomorphisms. Therefore $M-(\mathring V_1 \cup \mathring V_2)$ is an $h$-cobordism and we have a decomposition 
$$M = V_1 \cup_f V_2,$$
where $f \colon \partial V_1 =\#_r(S^n \times S^n) \to \partial V_2 = \#_r(S^n \times S^n)$ is a diffeomorphism. 
  
Now we look at the isomorphism $f_* \colon H_n(\#_r(S^n \times S^n)) \to H_n(\#_r(S^n \times S^n))$ induced by $f$. Let $\{e_1, \cdots , e_r, f_1, \cdots, f_r\}$ be the standard symplectic basis of $H_n(\#_r(S^n \times S^n))$ represented by $S_i^n \times \{*\}$ and $\{*\} \times S_i^n$, then by the construction of $V_1$ and $V_2$, the isomorphism $f_* \colon H_n(\#_r(S^n \times S^n)) \to H_n(\#_r(S^n \times S^n))$ is represented by the matrix 
$$\left ( \begin{array}{cc}
I & 0 \\
A & I \end{array} \right ) \in \mathrm{Sp}(2r, \mathbb Z).$$
In the segment of the Mayer-Vietoris sequence
$$H_n(\#_r(S^n \times S^n)) \stackrel{(i_{1*}, i_{2*})}{\longrightarrow} H_n(V_1) \oplus H_n(V_2) \to H_n(M) \to 0$$
under the basis $\{e_1, \cdots , e_r, f_1, \cdots, f_r\}$ of $H_n(\#_r(S^n \times S^n)) $ and the basis $\{x_1, \cdots , x_r , \bar x_1, \cdots \bar x_r \}$ of $H_n(V_1) \oplus H_n(V_2)$ the homomorphism $(i_{1*}, i_{2*})$ is represented by the matrix
$$\left ( \begin{array}{cc}
I & 0 \\
0 & A \end{array} \right ).$$
Since $H_n(M)$ is a free abelian group of rank $r$, we have $A=0$. Therefore the induced isomorphism $f_*$ is the identity. After an isotopy we may assume $f |_{\{*\} \times S^n_i}$ is the identity, then a basis of $H_{n+1}(M)$ is represented by $S^{n+1}_i=D^{n+1}_i \cup D^{n+1}_i$ ($i=1, \cdots, r$). Furthermore, the clutching function of the restriction of the stable tangent bundle of $M$ on $S^{n+1}_i$ is in the image of $\pi_n(SO(n)) \to \pi_n(SO)$, which is $\pi_n(SO)$ when $n>7$, or is a subgroup of index $2$ when $n=3$, $7$. (For more details see \cite{Kreck78}). It is known from \cite[Lemma 1.1]{K59} that for a stable vector bundle $\xi$ over $S^{n+1}$, $p_{\frac{n+1}{4}}(\xi)$ is $a_{\frac{n+1}{4}} ((n-1)/2)!$ multiple of the clutching function of $\xi$ under the identification $H^{n+1}(S^{n+1})=\pi_n(SO)$, where $a_m=2$ when $m$ is odd, and $a_m=1$ when $m$ is even. This gives the divisibility of $p_{\frac{n+1}{4}}(M)$ as claimed below Theorem \ref{classification}.

{\footnotesize \ }

{\footnotesize \
\bibliographystyle{alpha}
\bibliography{bib}
}
\Addresses
\end{document}